\documentclass[12pt]{article}
\usepackage{amssymb}
\usepackage{amsmath}
\usepackage{amsthm}
\usepackage{color}
\usepackage{geometry}		
\usepackage{graphicx}
\usepackage{xcolor}
\usepackage{tikz}							

\let\originalleft\left
\let\originalright\right
\renewcommand{\left}{\mathopen{}\mathclose\bgroup\originalleft}
\renewcommand{\right}{\aftergroup\egroup\originalright}

\newcommand*\filledCircled[1]{\tikz[baseline=(char.base)]{\node[fill=white,shape=circle,draw,inner sep=1.2pt](char){#1};}}

\begin{document}

\def\cA{\mathcal{A}}
\def\cB{\mathcal{B}}
\def\cC{\mathcal{C}}
\def\cD{\mathcal{D}}
\def\cX{\mathcal{X}}
\def\cY{\mathcal{Y}}
\def\ee{\varepsilon}

\newtheorem{theorem}{Theorem}[section]
\newtheorem{corollary}[theorem]{Corollary}
\newtheorem{lemma}[theorem]{Lemma}
\newtheorem{proposition}[theorem]{Proposition}
\newtheorem{conjecture}[theorem]{Conjecture}

\theoremstyle{definition}
\newtheorem{definition}{Definition}[section]
\newtheorem{example}[definition]{Example}

\theoremstyle{remark}
\newtheorem{remark}{Remark}[section]


\title{
Robust Chaos and the Continuity of Attractors.
}
\author{
P.A.~Glendinning$^{\dagger}$ and D.J.W.~Simpson$^{\ddagger}$\\
\small $^{\dagger}$School of Mathematics, University of Manchester, Oxford Road, Manchester, M13 9PL, UK.\\
\small $^{\ddagger}$School of Fundamental Sciences, Massey University, Palmerston North, New Zealand
}
\maketitle


\begin{abstract}

As the parameters of a map are varied an attractor may vary continuously in the Hausdorff metric.
The purpose of this paper is to explore the continuation of chaotic attractors.
We argue that this is not a helpful concept for smooth unimodal maps for which periodic windows fill parameter space densely,
but that for piecewise-smooth maps it provides a way to
delineate structure within parameter regions of robust chaos
and form a stronger notion of robustness.
We obtain conditions for the continuity of an attractor
and demonstrate the results with coupled skew tent maps, the Lozi map, and the border-collision normal form.


\end{abstract}

\section{Introduction}
\label{sec:intro}

Let $f_\mu$ be a family of maps on $\mathbb{R}^p$
that vary continuously with respect to a parameter $\mu \in M \subset \mathbb{R}^q$,
where $M$ is compact with non-empty interior.
We say that $f_\mu$ exhibits {\em robust chaos} in $M$ if $f_\mu$ has a chaotic attractor for each $\mu \in M$
and there exist $\mu_1, \mu_2 \in M$ such that $f_{\mu_1}$ is not topologically conjugate to $f_{\mu_2}$ \cite{BaYo98,Gl17}.
This last stipulation prohibits the trivial case that $f_\mu$ undergoes no topological change as $\mu$ is varied.
While robust chaos does not occur for generic smooth maps of the interval \cite{Va10},
it appears to be typical for maps that are piecewise-smooth \cite{GlSi19}.

In applications that utilise chaos, such as mixing \cite{Ot90},
spacecraft trajectories \cite{BoMe95},
and encryption \cite{KoLi11},
robust chaos is often a desired property.
It seems reasonable that chaotic attractors at nearby parameter values should be in some way related
because the map varies continuously even if the details of the dynamics can change.
The aim of this paper is to give mathematical meaning to this sense of sameness
by using continuity in the Hausdorff metric.
We show how this adds structure to parameter regimes of robust chaos,
and in fact a layered structure when multiple attractors coexist.

The continuity of attractors in the Hausdorff metric has been useful in a number of problems.
Stuart and Humphries \cite{StHu96} use it with the semi-distance (see \S\ref{sec:definitions})
to assess the numerical approximation of dynamical systems via the geometry of attractors rather than their dynamics.
This is a natural extension to continuation techniques for periodic orbits.
There are also fairly general results for the continuity of global attractors of semi-flows \cite{DeKl04}.
In particular Hoang {\em et.~al.}~\cite{HoOl15} develop a concise and effective characterisation of the continuity of global attractors
and we follow their approach in \S\ref{sec:UCC}.

Outside of \S\ref{sec:UCC} we investigate the continuity of attractors through a series of examples.
Our motivation is to develop ideas which can be used easily.
This, together with the fact that our definition of attractors (see \S\ref{sec:definitions}) is local,
means we need a slightly more complicated continuity argument than that of \cite{HoOl15}, although the underlying principles are the same.

\section{Definitions}
\label{sec:definitions}

Let $d$ be a metric on $\mathbb{R}^p$.
The (asymmetric) {\em semi-distance} between sets $X, Y \subseteq \mathbb{R}^p$ is defined as
\begin{equation}
\label{eq:asymdist}
d_a(X,Y) = \sup_{x \in X} \inf_{y \in Y} d(x,y),
\end{equation}
see Fig.~\ref{fig:schemHausdorffDistance} for a visualisation.
The Hausdorff distance is the following symmetric version of this semi-distance:
\begin{equation}
\label{eq:Hm}
d_H(X,Y) = \max \left[ d_a(X,Y), d_a(Y,X) \right].
\end{equation}
We also write
$$
B_r(X) = \left\{ y \in \mathbb{R}^p \,\middle|\, d_a(y,X) \le r \right\},
$$
to denote the closed ball of radius $r > 0$ around a set $X \subseteq \mathbb{R}^p$.

\begin{figure}[b!]
\begin{center}
\setlength{\unitlength}{1cm}
\begin{picture}(5.6,4.2)
\put(0,0){\includegraphics[height=4.2cm]{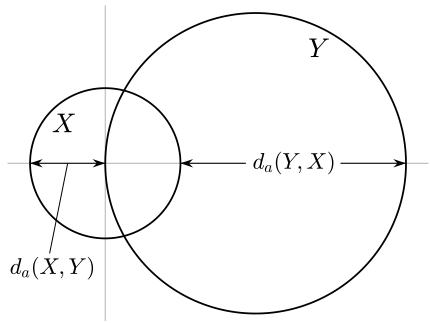}}
\end{picture}
\caption{
The semi-distance \eqref{eq:asymdist} for two closed discs in $\mathbb{R}^2$.
\label{fig:schemHausdorffDistance}
}
\end{center}
\end{figure}

\begin{definition}
\label{def:attr}
Let $f$ be a continuous map on $\mathbb{R}^p$.
A compact set $\cA \subset \mathbb{R}^p$ is an {\em attractor} of $f$ if
\begin{enumerate}
\setlength{\itemsep}{0pt}
\item
$f(\cA) = \cA$,
\item
$\cA$ contains a dense orbit, and
\item
there exists $r > 0$ such that
$d_a \left( f^n(x), \cA \right) \to 0$ 
as $n \to \infty$ for all $x \in B_r(\cA)$.
\end{enumerate}
If in addition Lyapunov exponents of typical points are positive then we say $\cA$ is a {\em chaotic} attractor.
\end{definition}

Now suppose a family of maps $f_\mu$ has an attractor $\cA_\mu$ for all $\mu \in M$.
To say that $\cA_\mu$ is continuous in the Hausdorff metric at some $\mu \in M$ means the following:
for all $\ee > 0$ there exists $\delta > 0$ such that
$d_H \left( \cA_\mu, \cA_\nu \right) < \ee$ whenever $\nu \in M$ and $|\mu - \nu| < \delta$.


\section{Tent maps}
\label{sec:tent}

Here we consider the tent map
\begin{equation}
\label{eq:tent}
T_s(x) = \begin{cases}
s x, & 0 \le x \le \frac{1}{2}, \\
s(1-x), & \frac{1}{2} \le x \le 1.
\end{cases}
\end{equation}
If $1 < s \le 2$ the tent map has a unique attractor on $[0,1]$,
see Fig.~\ref{fig:bifDiagSymmetricTentMap}.
If $\sqrt{2} \le s \le 2$, the attractor is the interval
$I_0(s) = \left[ s \left( 1 - \frac{s}{2} \right), \frac{s}{2} \right]$.
Otherwise it is the union of $2^n$ disjoint closed intervals where
$\sqrt{2} \le s^{2^{n-1}} < 2$, see \cite{Gl94,Va80}.
Lemma \ref{lemma:tent} below shows that, despite having different numbers of connected components,
the attractor is continuous in the Hausdorff metric.
In a similar way stable periodic solutions in period-doubling cascades are continuous
because, despite a change in the period at period-doubling bifurcations,
the attractor does not experience a jump in phase space.
In this and the next section we use $d(x,y) = |x-y|$.

\begin{figure}[b!]
\begin{center}
\setlength{\unitlength}{1cm}
\begin{picture}(12,4.2)
\put(0,0){\includegraphics[height=4.2cm]{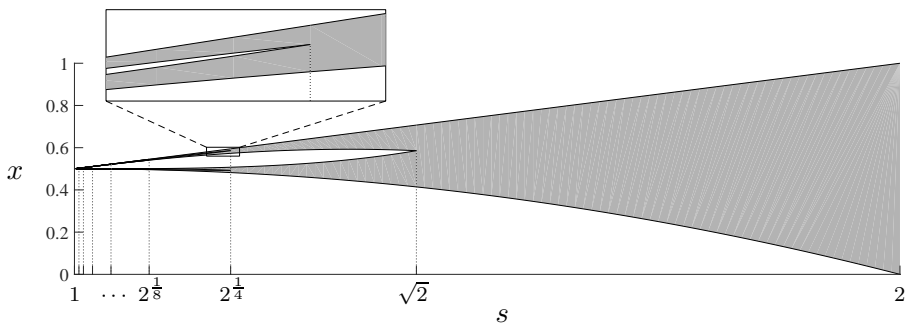}}
\put(6.5,0){\small $s$}
\put(0,1.9){\small $x$}
\put(.83,.27){\scriptsize $1$}
\put(1.25,.27){\scriptsize $\cdots$}
\put(1.75,.27){\scriptsize $2^{\frac{1}{8}}$}
\put(2.83,.27){\scriptsize $2^{\frac{1}{4}}$}
\put(5.2,.27){\scriptsize $\sqrt{2}$}
\put(11.8,.27){\scriptsize $2$}
\end{picture}
\caption{
The attractor of the tent map \eqref{eq:tent} for different values of the parameter $s$.
As the value of $s$ is decreased the number of intervals that comprise the attractor doubles at
$s = 2^\frac{1}{2^n}$ 
for each $n \ge 1$.
These band-splittings do not produce a jump in the Hausdorff distance, Lemma \ref{lemma:tent}.
\label{fig:bifDiagSymmetricTentMap}
}
\end{center}
\end{figure}

\begin{lemma}
\label{lemma:tent}
The attractor of the tent map \eqref{eq:tent} is continuous for $1 < s \le 2$.
\end{lemma}

\begin{proof}
If $\sqrt{2} < s \le 2$ the attractor is
$I_0 = \left[ T_s^2 \left( \frac{1}{2} \right), T_s \left( \frac{1}{2} \right) \right]$
which is continuous since its endpoints vary continuously.

Next we verify continuity at $s = \sqrt{2}$.
As $s \to \sqrt{2}$ from above the attractor is $I_0(s)$
and simply converges to $I_0 \left( \sqrt{2} \right)$.
As $s \to \sqrt{2}$ from below the attractor is the disjoint union $I_1(s) \cup I_2(s)$ where
\begin{align*}
I_1(s) &= \left[ T_s^2 \left( \tfrac{1}{2} \right), T_s^4 \left( \tfrac{1}{2} \right) \right], &
I_2(s) &= \left[ T_s^3 \left( \tfrac{1}{2} \right), T_s \left( \tfrac{1}{2} \right) \right],
\end{align*}
see Fig.~\ref{fig:twoBandCobwebSymmetricTentMap}.
The lower endpoint of $I_1(s)$ and the upper endpoint of $I_2(s)$ converge to the endpoints of
$I_0 \left( \sqrt{2} \right)$,
so it remains to show that the size of the gap between $I_1(s)$ and $I_2(s)$
converges to zero as $s \to \sqrt{2}$.
Indeed a direct calculation produces
$$
T_s^3 \left( \tfrac{1}{2} \right) - T_s^4 \left( \tfrac{1}{2} \right) =
s (s - 1) \left( 1 - \tfrac{s^2}{2} \right),
$$
which vanishes at $s = \sqrt{2}$,
This shows that $d_H \left( I_0 \left( \sqrt{2} \right), I_1(s) \cup I_2(s) \right) \to 0$
as $s \to \sqrt{2}$ from below,
so the attractor of \eqref{eq:tent} is continuous at $s = \sqrt{2}$.

If $s < \sqrt{2}$ then $T_s^2$ restricted to $I_1(s)$ or $I_2(s)$
is a linear rescaling of $T_{s^2}$ restricted to $I_0 \left( s^2 \right)$
and so the proof can be completed by considering higher iterates
in an inductive fashion \cite{Gl94,Va80}.
\end{proof}

\begin{figure}[h!]
\begin{center}
\setlength{\unitlength}{1cm}
\begin{picture}(6,6)
\put(0,0){\includegraphics[height=6cm]{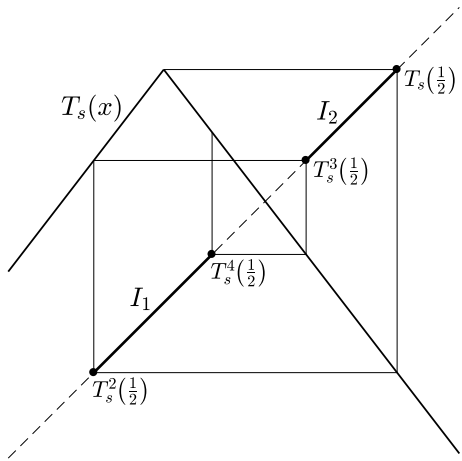}}
\end{picture}
\caption{
A cobweb diagram showing 
the two-band attractor of the tent map \eqref{eq:tent}
for $2^{\frac{1}{4}} < s < \sqrt{2}$.
\label{fig:twoBandCobwebSymmetricTentMap}
}
\end{center}
\end{figure}

\section{Quadratic maps}
\label{sec:quad}

The aim of this section is to argue that the continuation of chaotic attractors is not helpful for smooth maps
(at least in one dimension).
To do this we rely on a number of standard results for quadratic families which can be found, for example, in \cite{Va10,DeVa93}.
Smooth families of unimodal maps $f_\mu$ with a single maximum occurring at $c$ (the critical point)
can be described in terms of symbolic dynamics.
Every orbit defines a sequence of $C$'s, $L$'s and $R$'s by seeing if the
$n^{\rm th}$ point of the orbit equals $c$, lies to the left of $c$, or lies to the right of $c$, respectively.
There is a natural order on these sequences:
if $W$ is a finite sequence of symbols (a word) then
$WL < WC < WR$ if $W$ has an even number of $R$'s
and the inequality is reversed if $W$ has an odd number of $R$'s
(this is connected with the fact that $f^n$ is decreasing if there are an odd number of $R$'s).
The {\em kneading invariant} of $f_\mu$ is the symbol sequence $K_\mu$ associated with $f_\mu(c)$.
There are consistency conditions on the possible sequences that can occur:
if the critical point is not periodic then this condition is simply that $K_\mu$ is the maximal element of its shifts:
\begin{equation}\label{eq:consist}
\sigma^k \left( K_\mu \right) \le K_\mu \,, \quad \text{for all~} k \ge 1,
\end{equation}
where $\sigma$ is the shift map.
Families which depend smoothly ($C^1$) on the parameter are full if all possible consistent kneading invariants
between those of the end-points do actually occur in the family.
A family is monotonic if the kneading invariant is monotonic in the parameter
(this means in particular that certain behaviour is not repeated).
The standard quadratic maps are full and monotonic.
Further conditions such as convexity or negative Schwarzian derivative imply that there is a unique attractor at every parameter value.
Recall that $f_\mu$ is in a period-$n$ window if there exists $n>0$ and an interval $J$ such that $f_\mu^n$ restricted to $J$ is a unimodal map.

\begin{lemma}
\label{lemma:saddlenode}
Suppose that $f_\mu$ is a smooth unimodal map (at least $C^2$ in phase space and $C^1$ in the parameter) and has a unique attractor at each parameter value $\mu$ in some closed interval $I$. If there is a parameter $\mu_0\in \textrm{int}(I)$ with a non-degenerate saddle-node bifurcation of period $k>1$, then the attractor is not continuous in the Hausdorff metric at $\mu_0$.
\end{lemma}

\begin{proof}
We may assume without loss of generality that $f_{\mu_0}$ is not in a period $r$ window, $r<k$ (otherwise consider $f^r$).
In particular we can assume that the parameter lies in the single band region (the equivalent of $s>\sqrt{2}$ in the tent map)
with kneading invariant greater than $RLR^\infty$.
The kneading invariant at $\mu_0$ is $K(\mu_0)=(WC)^\infty$ with $W$ a sequence of $L$'s and $R$'s starting $RL$.
Choose the sign of $\mu$ so that the periodic orbit created by the saddle-node bifurcation exists if $\mu >\mu_0$. If $\mu<\mu_0$ with $|\mu -\mu_0|$ small, then the kneading invariant starts $(WE)^n\dots$,
with $n\to\infty$ as $\mu\to \mu_0$ from below and
$$
E =
\begin{cases}
L, & $W$ ~\text{has an even number of}~ R \text{'s}, \\
R, & $W$ ~\text{has an odd number of}~ R \text{'s}.
\end{cases}
$$
It is now an elementary exercise to show that the sequences $(WE)^n WLR^\infty$ if $E=L$ or $(WE)^nWR^\infty$
if $E=R$ satisfy the consistency conditions (\ref{eq:consist}), and since the points $c$ and its images define a Markov partition
(they are Misiurewicz points) and are not in a period $r$ window the attractor is the interval $I_{\mu_n}=[f^2_{\mu_n} (c),f_{\mu_n}(c)]$.
Thus arbitrarily close to the bifurcation value $\mu_0$ the attractor is an interval.

If $\mu =\mu_0$ then there is a non-hyperbolic periodic orbit $W_k$ of period $k$ which is the attractor for
$f$ and $k$ non-trivial close intervals $B_i$, $i=1, \dots ,k$, which are the immediate basins of attraction of the periodic orbit
of period $k$ (one end point of each of these intervals is a point of period $k$; these are one-sided basins of attraction).

We have already seen that in any small neighbourhood of $\mu_0$ there exists $\mu$
such that the attractor of $f_\mu$ is an interval $I_{\mu_n}$,
and there exists $i$ such that $B_i\subset I_{\mu_n}$ (as the period $k\ge 3$).
Hence if $|B_i|=2\ee$, $d_a(I_{\mu_n} ,W_k)>\ee$ since the closest that $W_k$ can be to the part of the attractor inside $B_i$ is $\ee$.
Hence the attractor is not continuous at $\mu_0$.
\end{proof}

\begin{corollary}
The only non-trivial intervals on which the attractor of the logistic map is continuous
in the Hausdorff metric are the period-doubling cascades of stable periodic orbits.
\end{corollary}

\begin{proof}
Between any two topologically distinct chaotic attractors
there exist parameter values with saddle-node bifurcations of periodic orbits.
\end{proof}

\section{Uniform continuity and continuation}
\label{sec:UCC}

In this section we derive general results for the continuity of attractors.
Our approach follows Hoang {\em et.~al.}~\cite{HoOl15}
with some technical additions required to accommodate local attractors
that will be useful when we come to the Lozi map in \S\ref{sec:Lozi}. 
 
Let $f_\mu$ be a family of maps on $\mathbb{R}^p$
where $\mu \in M$ and $M \subset \mathbb{R}^q$ is compact.
Assume $f_\mu$ varies continuously in phase space and in $\mu$.
Assume that for each $\mu \in M$, the map $f_\mu$ has an attractor $\cA_\mu$
and let $r_\mu > 0$ be a suitable value for part (iii) of Definition \ref{def:attr}.
Two further assumptions are needed.
\begin{itemize}
\item[(A1)]
There exists compact $\Omega \subset \mathbb{R}^p$ such that $B_{r_\mu}(\cA_\mu) \subseteq \Omega$ for all $\mu \in M$.
\item[(A2)]
For all $\mu \in M$ there exists compact $N_\mu \subseteq \Omega$,
continuous (with respect to $\mu$) in the Hausdorff metric,
such that $f(N_\mu) \subseteq N_\mu$ and
$\cA_\mu = {\rm cl} \left( \cap_{n=0}^\infty f^n_\mu (N_\mu) \right)$.
\end{itemize}
 
These are the natural generalizations of (L2) and (L3) of \cite{HoOl15}.
The following lemma shows that at each stage of the construction of the attractor by iterates of $N_\mu$,
the sets remain close (this is equivalent to Lemma 3.1 of \cite{HoOl15}).

\begin{lemma}
\label{lemma:close}
Suppose $f_\mu$ is continuous with an attractor $\cA_\mu$ and (A1) and (A2) hold.
For each $n \ge 0$, $f_\mu^n(N_\mu)$ is continuous in the Hausdorff metric in $M$.
\end{lemma}

\begin{proof}
Choose any $n \ge 0$ and $\ee > 0$.
Since $f^n$ is continuous in $x$ and $\mu$
and $\Omega$ and $M$ are compact,
by the Heine-Cantor theorem $f^n$ is uniformly continuous in $x$ and $\mu$.
Thus there exist $\delta_\Omega, \delta_M > 0$ such that
for all $x, y \in \Omega$ with $d(x,y) < \delta_\Omega$
and all $\mu, \nu \in M$ with $|\mu - \nu| < \delta_M$
we have $d \left( f_\mu^n(x), f_\nu^n(y) \right) < \ee$.

Since $N_\mu$ is continuous on the compact set $M$
it is similarly uniformly continuous and so there exists $\delta_1 > 0$ such that
for all $\mu, \nu \in M$ with $|\mu - \nu| < \delta_1$
we have $d_H \left( N_\mu, N_\nu \right) < \delta_\Omega$.

Let $\delta = \min(\delta_1,\delta_M)$.
Choose any $\mu, \nu \in M$ with $|\mu - \nu| < \delta$.
Then
$$
d_a \left( f_\mu^n(N_\mu), f_\nu^n(N_\nu) \right) =
\sup_{x \in N_\mu} \inf_{y \in N_\nu} d \left( f_\mu^n(x), f_\nu^n(y) \right) < \ee,
$$
because for all $x \in N_\mu$ there exists $y \in N_\nu$ such that $d(x,y) < \delta_\Omega$.
We similarly have $d_a \left( f_\nu^n(N_\nu), f_\mu^n(N_\mu) \right) < \ee$,
thus $d_H \left( f_\mu^n(N_\mu), f_\nu^n(N_\nu) \right) < \ee$, as required.
\end{proof}

\begin{theorem}
\label{th:cts}
If the conditions of Lemma \ref{lemma:close} hold and $d_H \left( f_\mu^n(N_\mu), \cA_\mu \right) \to 0$
as $n \to \infty$ uniformly in $M$, then $\cA_\mu$ is continuous in the Hausdorff metric in $M$.
\end{theorem}

\begin{proof}
Choose any $\ee > 0$.
There exists $n_0 \ge 0$ such that for all $\mu \in M$ and all $n \ge n_0$ we have
$d_H \left( f_\mu^n(N_\mu), \cA_\mu \right) < \frac{\ee}{3}$.
By Lemma \ref{lemma:close}, $f_\mu^n(N_\mu)$ is continuous in $M$,
but $M$ is compact so the continuity is uniform,
thus there exists $\delta > 0$ such that for all $\mu, \nu \in M$ with $|\mu - \nu| < \delta$
we have $d_H \left( f_\mu^n(N_\mu), f_\nu^n(N_\nu) \right) < \frac{\ee}{3}$.
Then for any $\mu, \nu \in M$ with $|\mu - \nu| < \delta$ we have
$$
d_H \left( \cA_\mu, \cA_\nu \right) \le
d_H \left( \cA_\mu, f_\mu^n(N_\mu) \right) +
d_H \left( f_\mu^n(N_\mu), f_\nu^n(N_\nu) \right) +
d_H \left( f_\nu^n(N_\nu), \cA_\nu \right) < \ee.
$$
\end{proof}

The most remarkable aspect of Hoang {\em et.~al}.~\cite{HoOl15}
is their proof that uniform convergence to the attractor with respect to the parameter
implies continuity of the attractor,
and, if the convergence is only pointwise then the continuity at least occurs on a residual set.
Recall, a {\em residual set} is the complement of a countable union of nowhere dense sets,
and every residual set is dense.
The uniform case is covered above by Theorem~\ref{th:cts},
so it remains for us to address pointwise convergence.
The following technical result will be needed, and indeed, contains all the hard work!

\begin{lemma}[Hoang {\em et.~al}.~\cite{HoOl15}]
\label{lemma:resid}
Let $\cX$ be a complete metric space, $\cY$ be a metric space,
and $g_n : \cX \to \cY$ be a family of continuous maps.
If the pointwise limit
$g(x) = \lim_{n \to \infty} g_n(x)$
exists for each $x \in \cX$, then $g$ is continuous on a residual subset of $\cX$.
\end{lemma}

In our case, $\cX$ is the parameter space $M$ and $\cY$ the space of compact subsets of $\mathbb{R}^p$ with the Hausdorff metric.

\begin{theorem}
\label{th:tech}
Suppose $f_\mu$ is continuous with an attractor $\cA_\mu$ and (A1) and (A2) hold.
Then $\cA_\mu$ is continuous in the Hausdorff metric on a residual subset of $M$.
\end{theorem}

\begin{proof}
Let $g_n(\mu) = f_\mu^n(N_\mu)$ and $g(\mu) = \cA_\mu$.
By Lemma \ref{lemma:close}, each $g_n(\mu)$ is continuous,
and by (A2), $g_n(\mu) \to g(\mu)$ as $n \to \infty$ for each $\mu \in M$,
so the result follows by Lemma \ref{lemma:resid}.
\end{proof}

\section{Coupled skew tent maps}
\label{sec:coupskew}

In the next three sections we identify continuous chaotic attractors
in three different piecewise-linear maps.
In these sections $d$ is the Euclidean metric on $\mathbb{R}^2$.

\begin{figure}[b!]
\begin{center}
\setlength{\unitlength}{1cm}
\begin{picture}(4.2,4.2)
\put(0,0){\includegraphics[height=4.2cm]{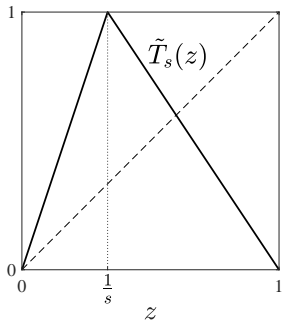}}
\end{picture}
\caption{
The skew tent map \eqref{eq:skewtent}.
\label{fig:schemSkewTentMap}
}
\end{center}
\end{figure}

Skew tent maps generalise \eqref{eq:tent} to allow two slopes that differ in absolute value.
Specifically we consider
\begin{equation}
\label{eq:skewtent}
\tilde{T}_s(z) =
\begin{cases} s z, & 0 \le z \le \frac{1}{s}, \\
\frac{s}{s-1}(1-z), & \frac{1}{s} \le z \le 1,
\end{cases}
\end{equation}
where $1 < s < 2$.
Each $\tilde{T}_s$, see Fig.~\ref{fig:schemSkewTentMap},
is a skew tent map on $[0,1]$ equivalent to a full shift on two symbols.
As considered originally in \cite{PiGr91},
here we use \eqref{eq:skewtent} to form the coupled skew tent map
\begin{equation}
\label{eq:coupledst}
f_{s,\omega}(x) = \begin{bmatrix}
(1-\omega) \tilde{T}_s(x_1) + \omega \tilde{T}_s(x_2) \\
\omega \tilde{T}_s(x_1) + (1-\omega) \tilde{T}_s(x_2)
\end{bmatrix},
\end{equation}
where $0 \le \omega \le \textstyle{\frac{1}{2}}$ is a measure of the coupling strength.
This is a map on $[0,1] \times [0,1]$ and we write $x = (x_1,x_2) \in \mathbb{R}^2$.

The diagonal $x_1 = x_2$ is an invariant set that is stable for sufficiently large values of $\omega$.
As the value of $\omega$ is decreased a `blowout bifurcation' occurs when typical transverse Lyapunov exponents become positive at
$\omega = \frac{1}{2} \left( 1 - {\rm e}^{-\gamma} \right)$,
where $\gamma = {\rm ln}(s) - \left( 1 - \frac{1}{s} \right) {\rm ln}(1-s)$, see \cite{Gl01}.
However, some orbits on the diagonal become transversely unstable
before the blowout bifurcation.
This first occurs at $\omega = \frac{1}{2 s}$ and is responsible for the creation of a two-dimensional attractor.

\begin{figure}[b!]
\begin{center}
\setlength{\unitlength}{1cm}
\begin{picture}(6,6)
\put(0,0){\includegraphics[height=6cm]{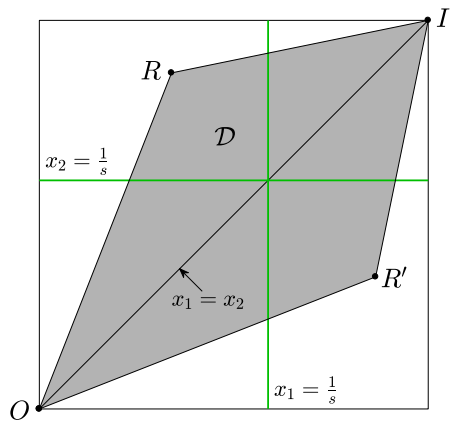}}
\end{picture}
\caption{
Phase space $(x_1,x_2)$ of the coupled skew tent map \eqref{eq:coupledst}
showing the quadrilateral $\mathcal{D}$ of Theorem \ref{th:Gl01}.
\label{fig:schemCoupledSkewQuad}
}
\end{center}
\end{figure}

\begin{theorem}[Glendinning \cite{Gl01}]
\label{th:Gl01}
Let $\frac{\sqrt{5} + 1}{2} < s < 2$.
Let $\mathcal{D}$ be the closed quadrilateral $ORIR'$ where
\begin{align*}
O &= (0,0), &
R &= \left( 2 \omega, \tfrac{1-2\omega+2\omega^2}{1-\omega} \right), &
I &= (1,1), &
R^\prime &= \left( \tfrac{1-2\omega+2\omega^2}{1-\omega}, 2 \omega \right),
\end{align*}
see Fig.~\ref{fig:schemCoupledSkewQuad}.
If $0 < \omega < \frac{1}{2 s}$
then $\mathcal{D}$ is the unique attractor of \eqref{eq:coupledst},
whilst if $\frac{1}{2 s} < \omega < \frac{1}{2}$ then the diagonal $x_1 = x_2$
is the unique attractor of \eqref{eq:coupledst}.
\end{theorem}

The two types of attractor: $\mathcal{D}$ and the diagonal $x_1 = x_2$,
are clearly chaotic and vary continuously with $s$ and $\omega$.
Consequently we have the following result.

\begin{corollary}
\label{cor:cstm}
Let $\frac{\sqrt{5} + 1}{2} < s < 2$.
Then \eqref{eq:coupledst} has robust chaos for $0 < \omega < \frac{1}{2}$.
The attractor is continuous in the Hausdorff metric for $0 < \omega < \frac{1}{2 s}$
and $\frac{1}{2 s} < \omega < \frac{1}{2}$.
\end{corollary}

The region of robust chaos, Fig.~\ref{fig:bifSetCoupledSkew},
is thus divided into two pieces by the curve $\omega = \frac{1}{2 s}$
through which the attractor cannot be continued.
In this way our consideration of continuity in the Hausdorff metric
has allowed us to partition the region of robust chaos into two different types in a formal way.

It could be objected that this example is a boundary case
as $\mathcal{D}$ does not satisfy part (iii) of Definition \ref{def:attr}.
In this sense the map has the same status as $x \mapsto 4 x (1-x)$
for which the interval $[0,1]$ is the `attractor' although all points outside this interval diverge.
This is a technical nicety that we expect can be circumvented by generalising the skew tent map \eqref{eq:skewtent} to
\begin{equation}
\label{eq:artent}
\tilde{T}_{s,t}(z) =
\begin{cases}
s z, & 0 \le z \le t, \\ 
\frac{s t}{1-t}(1-z), & t \le z \le 1,
\end{cases}
\end{equation}
where $0 < t < \frac{1}{s}$.
Numerical experiments suggest that the two-dimensional map obtained by replacing $\tilde{T}_s$
with $\tilde{T}_{s,t}$ in \eqref{eq:coupledst} exhibits an analogous continuous quadrilateral attractor
that now satisfies part (iii) of Definition \ref{def:attr} for some $r > 0$,
but it remains to carefully extend the construction of $\mathcal{D}$ given in \cite{Gl01} to allow $t < \frac{1}{s}$.

\begin{figure}[h!]
\begin{center}
\setlength{\unitlength}{1cm}
\begin{picture}(5.6,4.2)
\put(0,0){\includegraphics[height=4.2cm]{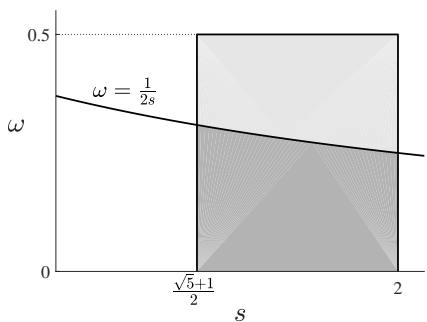}}
\end{picture}
\caption{
A parameter region of the coupled skew tent map \eqref{eq:coupledst} corresponding to robust chaos.
Above $\omega = \frac{1}{2 s}$ the attractor is the diagonal $x_1 = x_2$;
below this curve the attractor is $\mathcal{D}$,
see Fig.~\ref{fig:schemCoupledSkewQuad} and Corollary \ref{cor:cstm}.	
\label{fig:bifSetCoupledSkew}
}
\end{center}
\end{figure}

\section{Lozi Maps}
\label{sec:Lozi}

The Lozi map \cite{Lo78}
\begin{equation}
\label{eq:Lozi}
L(x) = \begin{bmatrix}
1 - a|x_1| + x_2 \\
b x_1
\end{bmatrix},
\end{equation}
where $a, b \in \mathbb{R}$ are parameters,
is a piecewise-linear version of the H\'{e}non map.
Misiurewicz established robust chaos for \eqref{eq:Lozi} in \cite{Mi80}.

\begin{theorem}[Misiurewicz \cite{Mi80}]
\label{th:Mi80}
Suppose
\begin{align}
\label{eq:lpars}
0 &< b < 1, &
0 &< a < \tfrac{4-b}{2}, &
a &> \tfrac{b + 2}{\sqrt{2}}, &
b &< \tfrac{a^2 - 1}{2 a + 1}.
\end{align}
Then the Lozi map \eqref{eq:Lozi} has a unique saddle-type fixed point in $x_1 > 0$ (denoted $X$)
and the closure of the unstable manifold of this point is a chaotic attractor $\cA$.
\end{theorem}

Here we adapt Misiurewicz's construction to show that the chaotic attractor he obtains
varies continuously with $a$ and $b$.
Our proof uses the results of \S\ref{sec:UCC}
and explains why it was necessary to add the variation
of the fundamental converging sets $N_\mu$ in that section.


\begin{theorem}
Throughout the parameter region \eqref{eq:lpars}
the attractor $\cA$ of Theorem \ref{th:Mi80} is continuous in the Hausdorff metric.
\end{theorem}

\begin{figure}[b!]
\begin{center}
\setlength{\unitlength}{1cm}
\begin{picture}(8,6)
\put(0,0){\includegraphics[height=6cm]{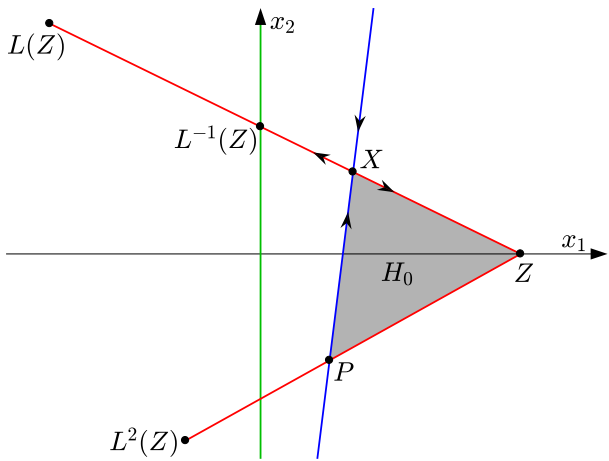}}
\end{picture}
\caption{
Parts of the stable (blue) and unstable (red) manifolds of the saddle-type fixed point $X$
of the Lozi map \eqref{eq:Lozi}.
Note, the stable eigenvalue associated with $X$ is positive
(so the stable manifold has two dynamically independent branches),
while the unstable eigenvalue is negative
(so the unstable manifold has one dynamically independent branch).
\label{fig:schemPhaseSpaceLozi}
}
\end{center}
\end{figure}

\begin{proof}
Following \cite{Mi80},
let $X$ be the fixed point in $x_1 > 0$,
let $Z$ be the intersection of the local unstable manifold of $X$ with the $x_1$-axis,
and let $P$ be the intersection of the local stable manifold of $X$ with the line segment $Z L^2(Z)$,
see Fig.~\ref{fig:schemPhaseSpaceLozi}.
Let $H_0$ be the compact filled triangle $XZP$.
The conditions \eqref{eq:lpars} imply that
$H_0$ is contained in the region $x_1 > 0$,
the line segments $X Z$ and $Z P$ belong to the unstable manifold of $X$,
and $X P$ belongs to the local stable manifold of $X$.
It follows that every point on the boundary of the forward invariant set $H = \cup_{k=0}^\infty L^k(H_0)$
belongs to either the unstable manifold of $X$ or the line segment $X P$.
Moreover every point on the boundary of $L^n(H)$
belongs to either the unstable manifold of $X$ or the line segment $X L^n(P)$.
Notice $d \left( X, L^n(P) \right) = \lambda_s^n d(X,P)$,
where $0 < \lambda_s < 1$ is the stable eigenvalue associated with $X$.

Misiurewicz \cite{Mi80} shows that $\cap_{k=0}^\infty L^k(H)$ is the attractor $\cA$ of Theorem \ref{th:Mi80}.
By Theorem~\ref{th:cts} it remains to show that $d_H \left( L^n(H), \cA \right) \to 0$ as $n \to \infty$
uniformly in $a$ and $b$.

The compact filled triangle $Z L(Z) L^2(Z)$ is forward invariant, see \cite{Mi80},
so if $K$ denotes the area of this triangle then ${\rm Area}(H) \le K$.
For each $n$, ${\rm Area} \left( L^n(H) \right) = b^n {\rm Area}(H)$
(because $L$ is invertible and the absolute value of the
determinant of the Jacobian matrix of $L$ is $b$ at all points with $x_1 \ne 0$).
Thus the distance of any $x \in L^n(H)$ to the boundary of $L^n(H)$
is at most $\sqrt{\frac{K}{\pi}} \,b^{\frac{n}{2}}$
(obtained by imagining $L^n(H)$ as a circle with centre $x$ and using the Euclidean metric).
Thus for any $x \in L^n(H)$,
$$
d_a \left( L^n(H), \cA \right) \le {\textstyle \sqrt{\frac{K}{\pi}}} \,b^{\frac{n}{2}} + \lambda_s^n d(X,P).
$$
and since $\cA \subseteq L^n(H)$ the same bound applies to $d_H \left( L^n(H), \cA \right)$.

Now fix any pair of parameters $(a_0,b_0) \in \mathbb{R}^2$ satisfying \eqref{eq:lpars}.
There exists $\delta > 0$ such that \eqref{eq:lpars} is satisfied by
all $(a,b) \in \mathbb{R}^2$ a distance at most $\delta$ from $(a_0,b_0)$,
call this parameter set $M$.
Denote the supremum values of $K$, $b$, $\lambda_s$, and $d(X,P)$ over $M$
by $K_{\rm max}$, $b_{\rm max}$, $\lambda_{s,{\rm max}}$, and $d_{\rm max}$, respectively.
Then for any $(a,b) \in M$ we have
$$
d_H \left( L^n(H), \cA \right) \le {\textstyle \sqrt{\frac{K_{\rm max}}{\pi}}} \,b_{\rm max}^{\frac{n}{2}} + \lambda_{s,{\rm max}}^n d_{\rm max}.
$$
Since $b_{\rm max}, \lambda_{s,{\rm max}} < 1$
we conclude that $d_H \left( L^n(H), \cA \right) \to 0$ as $n \to \infty$ uniformly in $M_\delta$.
Thus $\cA$ is continuous at $(a_0,b_0)$ by Theorem~\ref{th:cts}.
\end{proof}

\section{Border-collision normal form}
\label{sec:bcnf}

In this section we describe a numerical example of bifurcations of continuous chaotic attractors in
the two-dimensional border-collision normal form
\begin{equation}
\label{eq:bcnf}
x \mapsto
\begin{cases}
\begin{bmatrix} \tau_L & 1 \\ -\delta_L & 0 \end{bmatrix} x +
\begin{bmatrix} 1 \\ 0 \end{bmatrix}, & x_1 \le 0, \\
\begin{bmatrix} \tau_R & 1 \\ -\delta_R & 0 \end{bmatrix} x +
\begin{bmatrix} 1 \\ 0 \end{bmatrix}, & x_1 \ge 0,
\end{cases}
\end{equation}
which has parameters $\tau_L, \delta_L, \tau_R, \delta_R \in \mathbb{R}$.
This map, introduced in \cite{NuYo92},
is a generalisation of the Lozi map and can be used to approximate the dynamics
near any generic border-collision bifurcation in two dimensions \cite{Si16}.

\begin{figure}[b!]
\begin{center}
\setlength{\unitlength}{1cm}
\begin{picture}(12,6)
\put(0,0){\includegraphics[height=6cm]{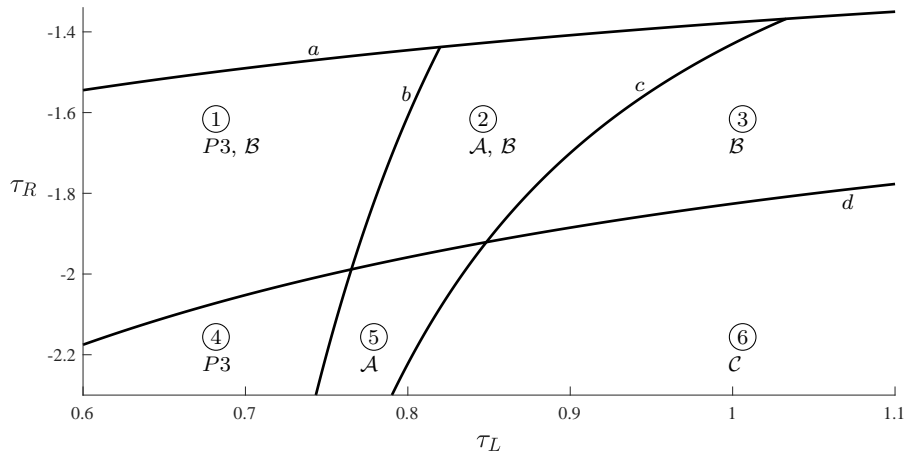}}
\put(4,5.34){\scriptsize $a$}
\put(5.25,4.72){\scriptsize $b$}
\put(8.35,4.86){\scriptsize $c$}
\put(11.1,3.295){\scriptsize $d$}
\put(2.6,4.4){\scriptsize $\filledCircled{1}$}
\put(2.6,4.03){\scriptsize $P3,\,\mathcal{B}$}
\put(6.15,4.4){\scriptsize $\filledCircled{2}$}
\put(6.15,4.03){\scriptsize $\mathcal{A},\,\mathcal{B}$}
\put(9.6,4.4){\scriptsize $\filledCircled{3}$}
\put(9.6,4.03){\scriptsize $\mathcal{B}$}
\put(2.6,1.5){\scriptsize $\filledCircled{4}$}
\put(2.6,1.13){\scriptsize $P3$}
\put(4.7,1.5){\scriptsize $\filledCircled{5}$}
\put(4.7,1.13){\scriptsize $\mathcal{A}$}
\put(9.6,1.5){\scriptsize $\filledCircled{6}$}
\put(9.6,1.13){\scriptsize $\mathcal{C}$}
\end{picture}
\caption{
A two-dimensional slice of the parameter space of the two-dimensional border-collision normal form \eqref{eq:bcnf}
defined by the restriction $\delta_L = \delta_R = 0.3$, \eqref{eq:AvSc12_deltaLdeltaR}.  The stable period-three orbit ($LRL$-cycle) is labelled $P3$; the chaotic attractors are labelled $\mathcal{A}$, $\mathcal{B}$ and $\mathcal{C}$. The bifurcation curves are $a$: border collision; $b$: non-smooth period-doubling; $c$: boundary crisis of attractor $\mathcal{A}$; and $d$: boundary crisis of attractor $\mathcal{B}$.
\label{fig:coexistChaosAvSc12_BifSet}
}
\end{center}
\end{figure}

We develop an example of \cite{AvSc12} and fix
\begin{align}
\delta_L &= 0.3, &
\delta_R &= 0.3.
\label{eq:AvSc12_deltaLdeltaR}
\end{align}
In the $(\tau_L,\tau_R)$-plane, see Fig.~\ref{fig:coexistChaosAvSc12_BifSet},
four codimension-one bifurcation curves, labelled $a$--$d$ and explained below,
divide parameter space into six regions, labelled 1--6.
Fig.~\ref{fig:coexistChaosAvSc12_qqAll} provides one representative phase portrait for each region.
Numerically we observe three continuous chaotic attractors,
a three or six-piece attractor $\cA$ (purple) in regions 2 and 5,
a one-piece attractor $\cB$ (yellow) in regions 1--3,
and a merging of these two attractors $\cC$ (cyan) in region 6.

Let us first describe the four bifurcation curves. Curve $a$ is the locus of a border collision bifurcation. Below curve $a$ there exist unique $LRL$ and $RRL$-cycles (these are period-$3$ solutions with the indicated symbolic itineraries \cite{Si16}).
The $LRL$-cycle is stable in regions 1 and 4. On curve $b$ the $LRL$-cycle has an eigenvalue of $-1$ and there exists a period-$6$ solution with one point on the switching manifold $x_1 = 0$. This solution grows continuously into attractor $\cA$ in regions 2 and 5.
As $\tau_L$ increases, crossing curve $b$, there is a transition from the stable $LRL$-cycle to $\cA$ which is not continuous in the Hausdorff metric because the $LRL$-cycle and period-$6$ solution do not coincide on curve $b$.
For a greater description of this type of non-smooth period-doubling bifurcation refer to \cite{Si10,SuGa10}. Curves $c$ and $d$ are the loci of boundary crisis bifurcations which create and destroy the attractor $\cA$ (curve $c$) and $\cB$ (curve $d$). The intersection of these two curves is a codimension two boundary crisis described by \cite{Os06}, and these curves form the boundary of the region in which attractor $\cC$ exists.

\begin{figure}[b!]
\begin{center}
\setlength{\unitlength}{1cm}
\begin{picture}(16,11)
\put(0,6){\includegraphics[height=5cm]{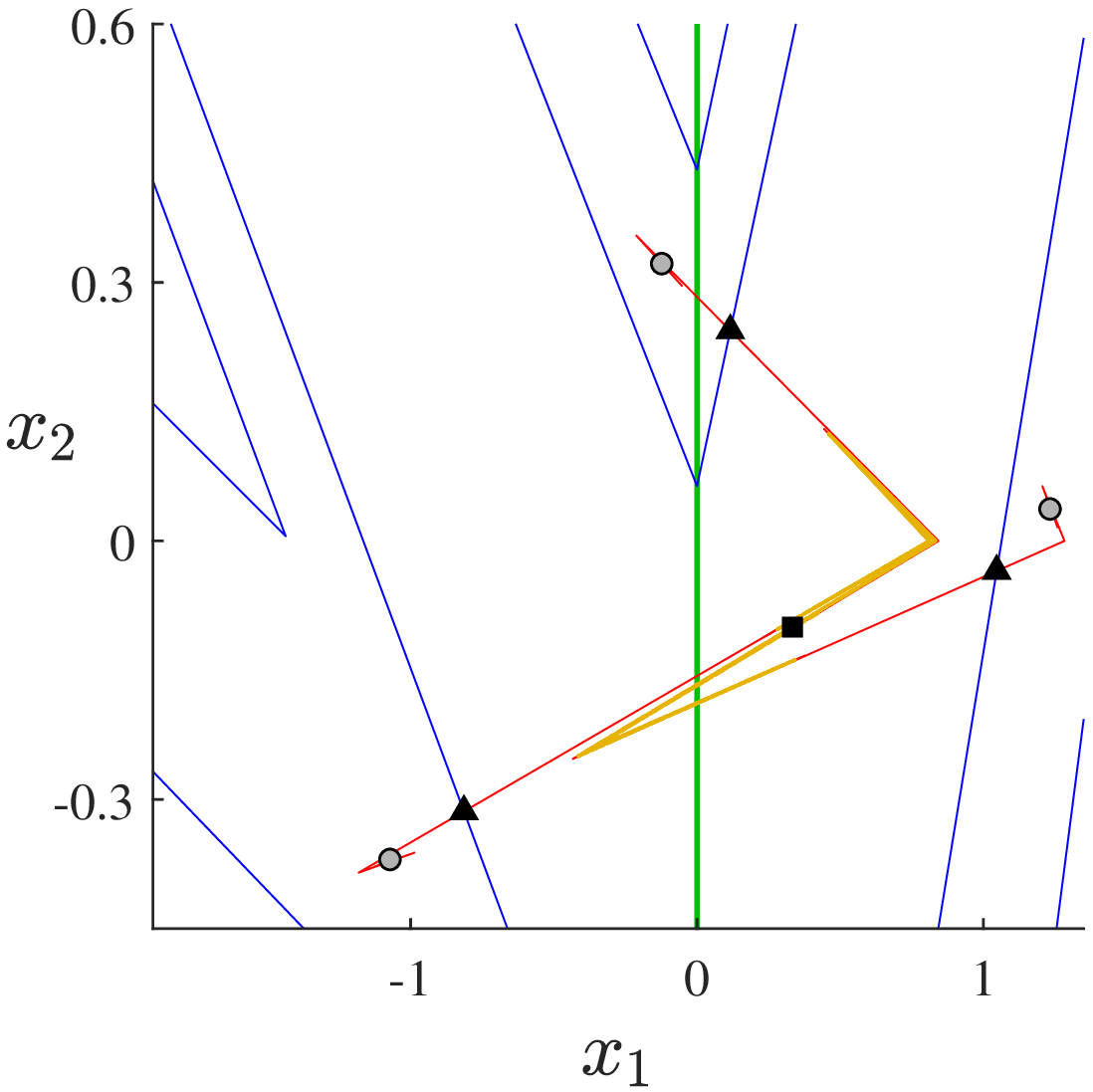}}
\put(5.5,6){\includegraphics[height=5cm]{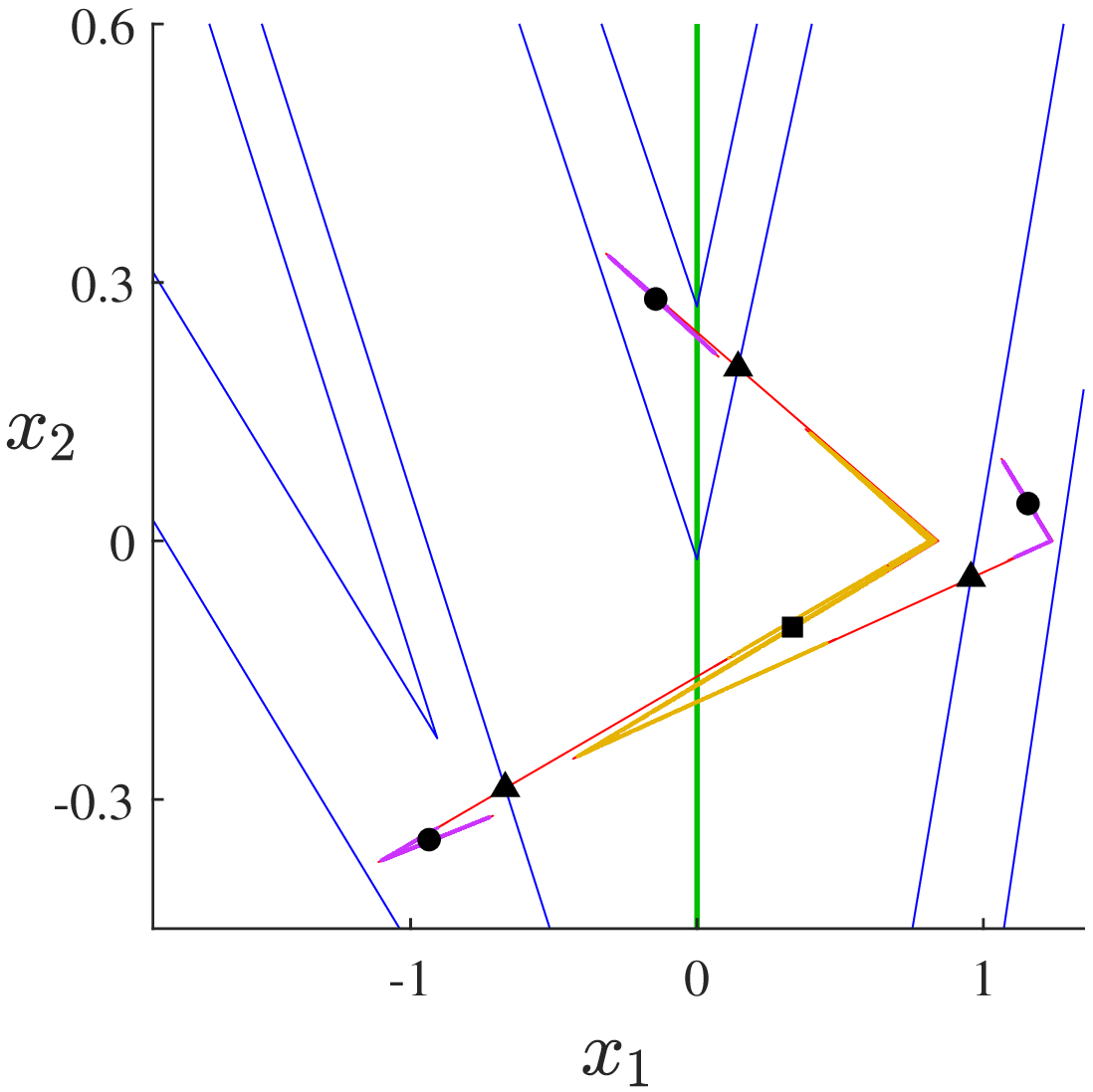}}
\put(11,6){\includegraphics[height=5cm]{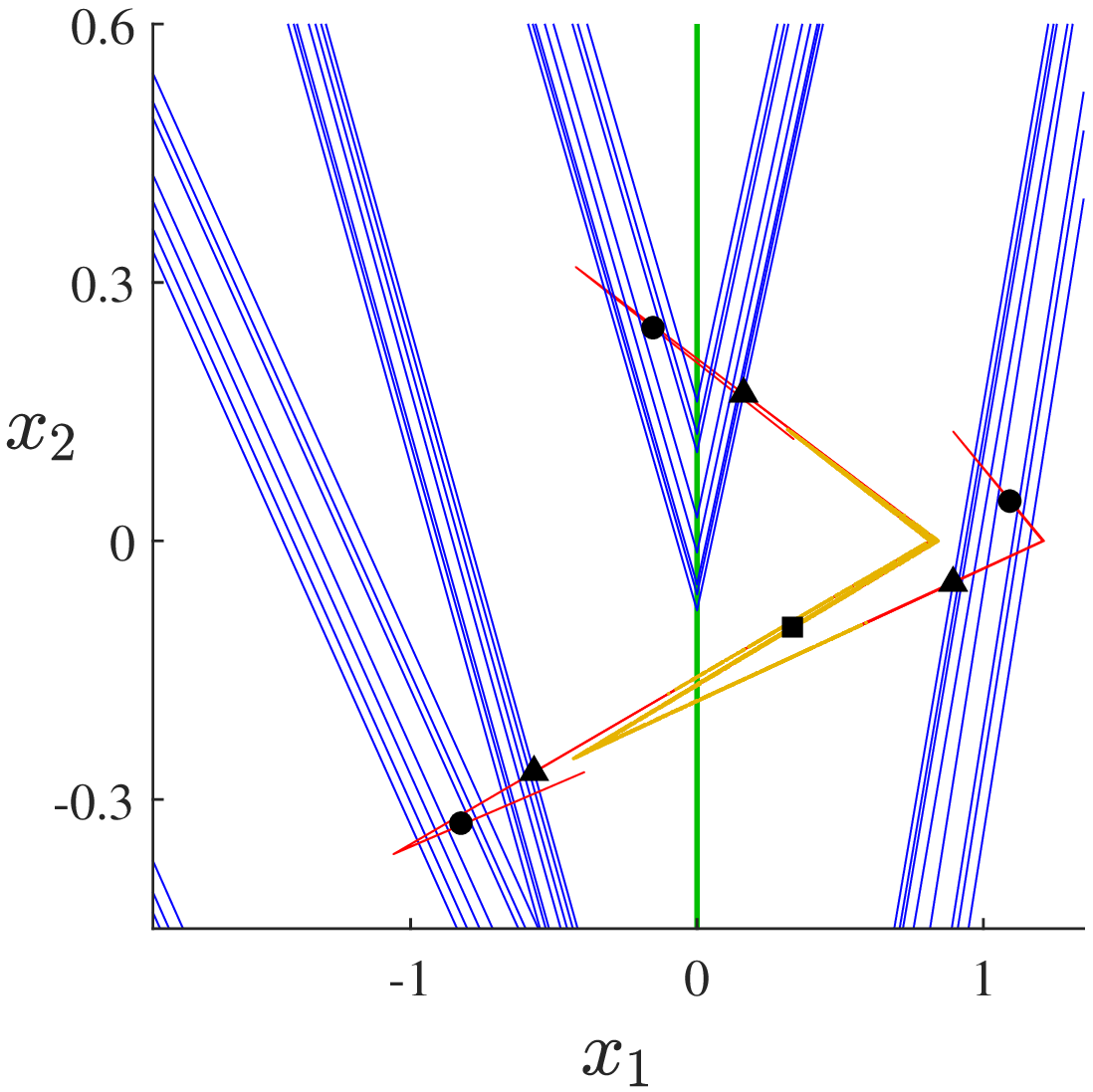}}
\put(0,0){\includegraphics[height=5cm]{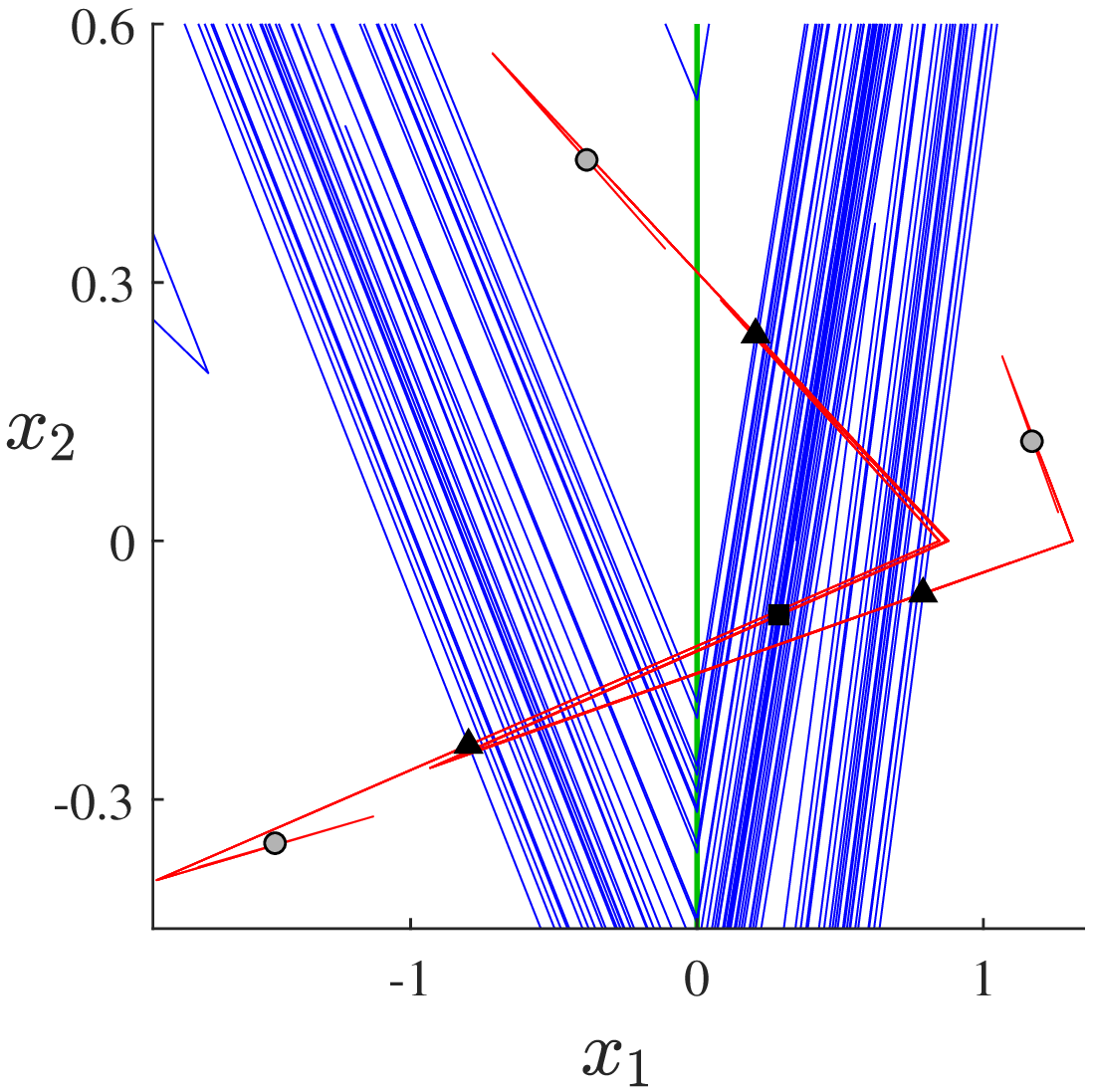}}
\put(5.5,0){\includegraphics[height=5cm]{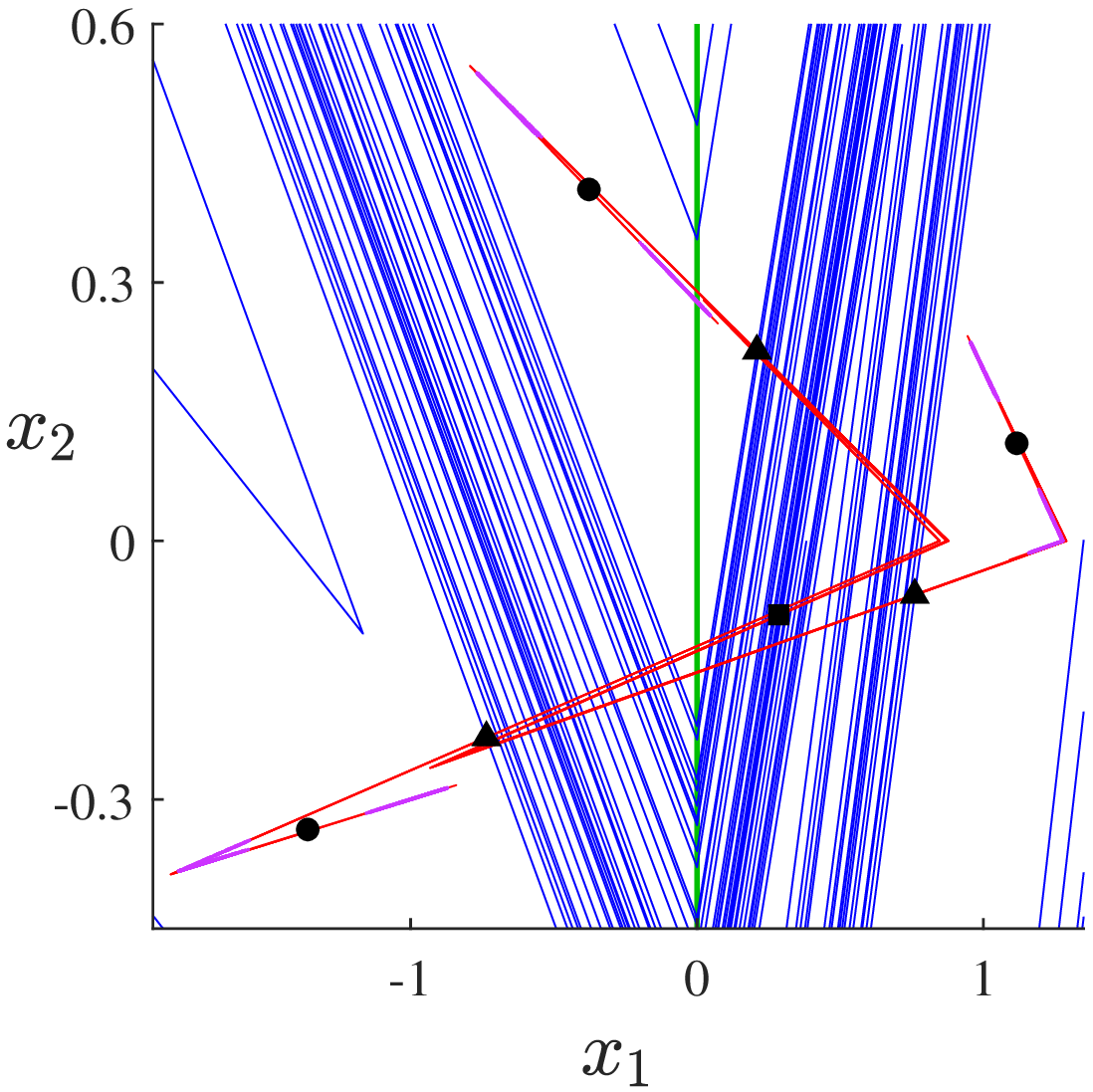}}
\put(11,0){\includegraphics[height=5cm]{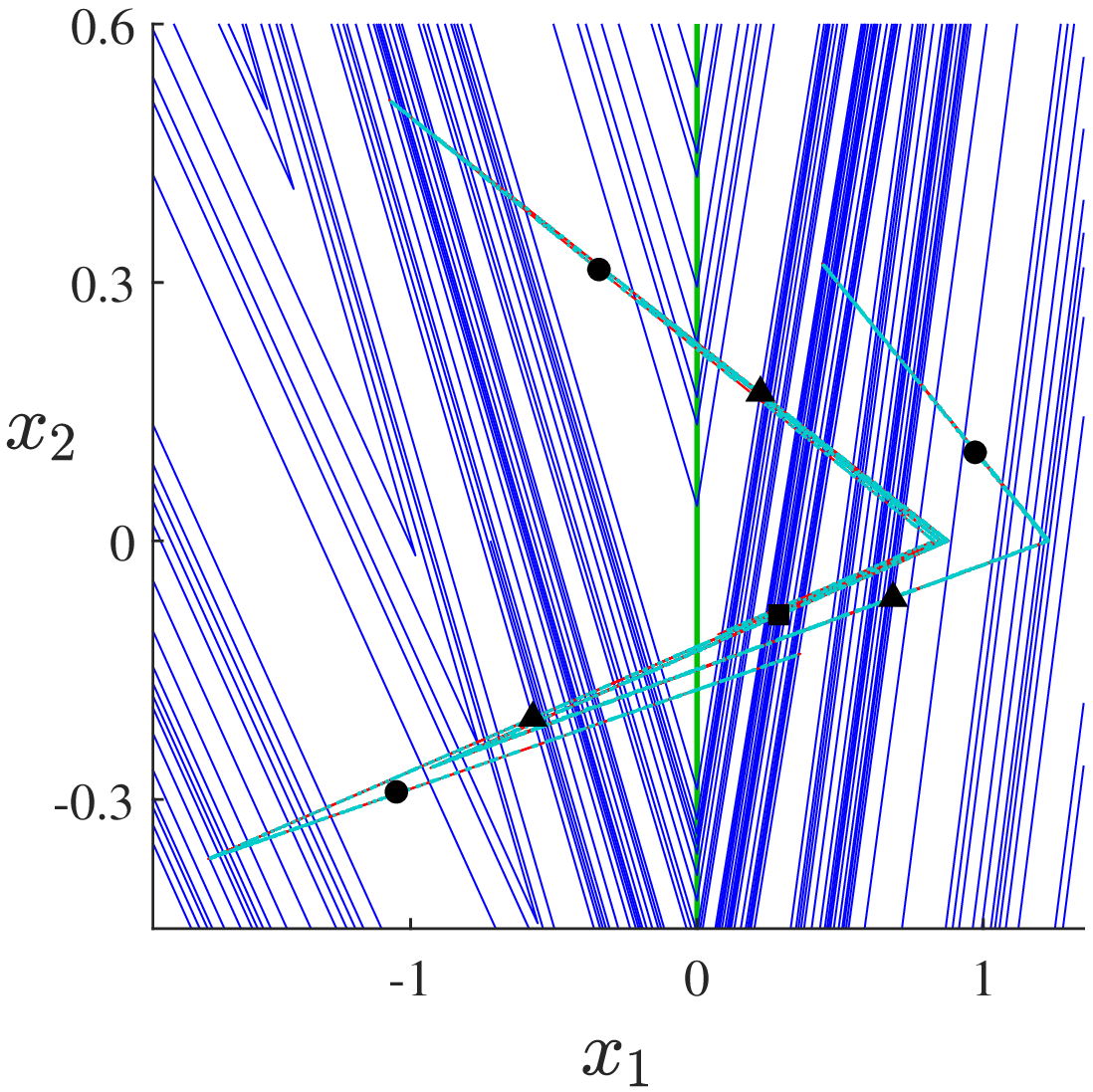}}
\put(.86,10.7){\scriptsize $\filledCircled{1}$}
\put(6.36,10.7){\scriptsize $\filledCircled{2}$}
\put(11.86,10.7){\scriptsize $\filledCircled{3}$}
\put(.86,4.7){\scriptsize $\filledCircled{4}$}
\put(6.36,4.7){\scriptsize $\filledCircled{5}$}
\put(11.86,4.7){\scriptsize $\filledCircled{6}$}
\end{picture}
\caption{
Phase portraits of \eqref{eq:bcnf} with \eqref{eq:AvSc12_deltaLdeltaR}
for sample parameter values in regions 1--6 of Fig.~\ref{fig:coexistChaosAvSc12_BifSet}.
The fixed point in $x_1 > 0$ is shown with a square,
the $LRL$-cycle is shown with circles,
and the $RRL$-cycle is shown with triangles.
The stable and unstable manifolds of the $RRL$-cycle, $W^s$ and $W^u$, are shown blue and red respectively
(these were computed by numerically growing the manifolds outwards
from the $RRL$-cycle for a large number of iterations).
The chaotic attractors $\cA$, $\cB$, and $\cC$ are coloured purple, yellow, and cyan respectively.
\label{fig:coexistChaosAvSc12_qqAll}
}
\end{center}
\end{figure}

In all six regions the $RRL$-cycle is a saddle
and its stable and unstable manifolds, $W^s$ and $W^u$,
are shown in Fig.~\ref{fig:coexistChaosAvSc12_qqAll}.
In regions 1 and 2, $W^s$ forms the boundary between the basins of attraction of the two coexisting attractors.
The unstable eigenvalue associated with the $RRL$-cycle is positive
so $W^u$ has two dynamically independent branches (and each branch has three pieces).

Points on the `outer' branch of $W^u$ converge (under forward iteration of \eqref{eq:bcnf})
to the stable $LRL$-cycle in regions 1 and 4
and to the attractor $\cA$ in regions 2 and 5.
The attractor $\cA$ is destroyed in a {\em crisis} on curve $c$:
here the outer branch of $W^u$ attains an intersection with $W^s$.
This is a first homoclinic tangency \cite{PaTa93} except $W^s$ and $W^u$ are piecewise-linear
so form `corner' intersections \cite{Si16b}.
To the right of curve $c$ points on the outer branch converge to the same attractor as points on the inner branch.

In regions 1--3, points on the `inner' branch of $W^u$ converge to the attractor $\cB$.
This attractor is destroyed in a crisis on curve $d$:
here the inner branch of $W^u$ attains an intersection with $W^s$.
Below curve $d$ points on the inner branch converge to the same attractor as points on the outer branch.

In region 6 points converge to attractor $\cC$ which involves
both parts of phase space associated with $\cA$ and $\cB$.
As we cross curves $c$ or $d$ the transition from $\cA$ or $\cB$ to $\cC$ is not continuous in the Hausdorff metric
because the crises cause orbits to suddenly access new areas of phase space.

As an additional visualisation, Fig.~\ref{fig:coexistChaosAvSc12_Lyap} shows numerically computed
maximal Lyapunov exponents of the attractors.
The observation that the Lyapunov exponents of $\cA$, $\cB$, and $\cC$ are positive
and vary continuously in their respective regions supports our conjecture that
these attractors are chaotic and continuous.
The Lyapunov exponent varies continuously as we cross from region 6 to region 3 through curve $d$
because as we approach curve $d$ the fraction of iterates of attractor $\cC$
that dwell near attractor $\cB$ tends to $1$
(the invariant measure changes continuously across curve $d$),
and similarly from region 6 to region 5 through curve $c$.

\begin{figure}[t!]
\begin{center}
\setlength{\unitlength}{1cm}
\begin{picture}(12,6)
\put(0,0){\includegraphics[height=6cm]{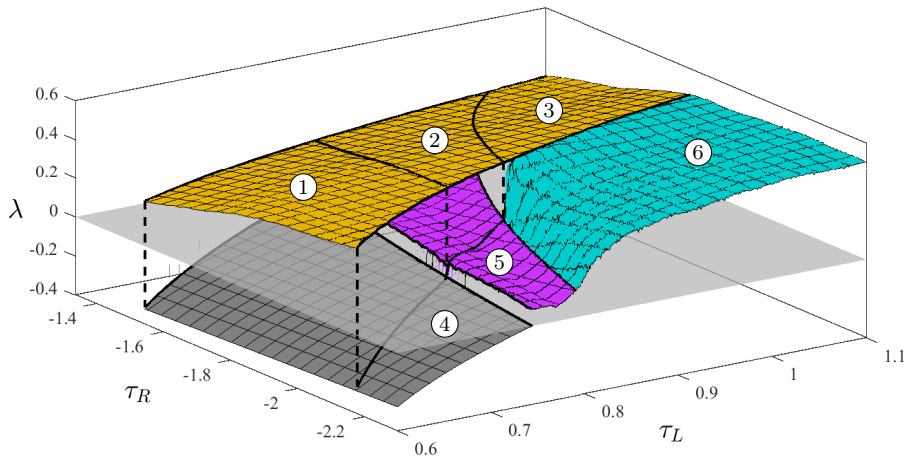}}
\put(3.8,3.5){\scriptsize $\filledCircled{1}$}
\put(5.55,4.12){\scriptsize $\filledCircled{2}$}
\put(7.07,4.52){\scriptsize $\filledCircled{3}$}
\put(5.68,1.68){\scriptsize $\filledCircled{4}$}
\put(6.41,2.5){\scriptsize $\filledCircled{5}$}
\put(9.05,3.96){\scriptsize $\filledCircled{6}$}
\end{picture}
\caption{
Numerically computed maximal Lyapunov exponents for attractors $\cA$ (purple),
$\cB$ (yellow), $\cC$ (cyan), and the $LRL$-cycle (grey).
The bifurcation curves of Fig.~\ref{fig:coexistChaosAvSc12_BifSet} have been overlaid for reference.
Regions 1 and 2 have coexisting attractors so here there are two surfaces of Lyapunov exponents.
\label{fig:coexistChaosAvSc12_Lyap}
}
\end{center}
\end{figure}

In summary, \eqref{eq:bcnf} has robust chaos in all but region 4
and each chaotic attractor appears to be continuous in the Hausdorff metric in the regions in which it exists.
The particular novelty of this example is region 2
where the chaotic attractors $\cA$ and $\cB$ coexist.
One may continue each attractor separately,
$\cA$ may be continued into region 5,
while $\cB$ may be continued into regions 1 and 3.

\section{Discussion}
\label{sec:conc}

In this paper we have added depth to the phenomenon of robust chaos in piecewise-smooth maps.
Previous works have shown piecewise-smooth maps to exhibit robust chaos
in the sense that a chaotic attractor exists throughout an open region of parameter space.
In several examples we have found this attractor to be continuous in the Hausdorff metric
and in this sense exhibits an extra level of robustness.

In the context of numerical exploration
such an attractor could be continued numerically along a one-dimensional path in parameter space.
Given attractors at two different points in parameter space,
one could ask whether or not there exists a path along which
one attractor can be continued into the other.
We stress that the continuation of invariant sets and attracting sets is more commonplace.
Attractors are more restrictive objects needing, among other things, a dense orbit (see Definition \ref{def:attr}),
and so, as argued in \S\ref{sec:quad}, the continuation of a chaotic attractor
may only be useful for piecewise-smooth maps.

Rather than use the Hausdorff metric, one could instead consider the continuity
of an attractor with respect to its Lyapunov spectrum,
the topology of its support (e.g.~number of holes),
its invariant probability measure \cite{AlPu17},
or, in the case of piecewise-smooth maps,
the fraction of iterates that lie on one side of the switching manifold
(which may have a useful physical interpretation).
Indeed, as evident from Fig.~\ref{fig:coexistChaosAvSc12_Lyap},
if one continued attractors using the maximal Lyapunov exponent,
attractors $\cA$ and $\cB$ in region 2 could be connected by a closed path through regions 5, 6, and 3.


\begin{thebibliography}{10}

\bibitem{BaYo98}
S.~Banerjee, J.A. Yorke, and C.~Grebogi.
\newblock Robust chaos.
\newblock {\em Phys. Rev. Lett.}, 80(14):3049--3052, 1998.

\bibitem{Gl17}
P.~Glendinning.
\newblock Robust chaos revisited.
\newblock {\em Eur. Phys. J. Special Topics}, 226(9):1721--1738, 2017.

\bibitem{Va10}
S.~van Strien.
\newblock One-parameter families of smooth interval maps: {D}ensity of
  hyperbolicity and robust chaos.
\newblock {\em Proc. Amer. Math. Soc.}, 138(12):4443--4446, 2010.

\bibitem{GlSi19}
P.A. Glendinning and D.J.W. Simpson.
\newblock Constructing robust chaos: invariant manifolds and expanding cones.
\newblock {\em Submitted.}, 2019.

\bibitem{Ot90}
J.M. Ottino.
\newblock Mixing, chaotic advection, and turbulence.
\newblock {\em Annu. Rev. Fluid Mech.}, 22:207--253, 1990.

\bibitem{BoMe95}
E.M. Bollt and J.D. Meiss.
\newblock Targeting chaotic orbits to the {M}oon through recurrence.
\newblock {\em Phys. Lett. A}, 204:373--378, 1995.

\bibitem{KoLi11}
L.~Kocarev and S.~Lian, editors.
\newblock {\em Chaos-Based Cryptography. Theory, Algorithms and Applications}.
\newblock Springer, New York, 2011.

\bibitem{StHu96}
Stuart. A.M. and A.R. Humphries.
\newblock {\em Dynamical Systems and Numerical Analysis.}
\newblock Cambridge University Press, New York, 1996.

\bibitem{DeKl04}
L.~Desheng and P.E. Kloeden.
\newblock Equi-attraction and the continuous dependence of attractors on
  parameters.
\newblock {\em Glasgow Math. J.}, 46:131--141, 2004.

\bibitem{HoOl15}
L.T. Hoang, E.J. Olson, and J.C. Robinson.
\newblock On the continuity of global attractors.
\newblock {\em Proc. Amer. Math. Soc.}, 143:4389--4395, 2015.

\bibitem{Gl94}
P.~Glendinning.
\newblock {\em Stability, Instability and Chaos: An Introduction to the Theory
  of Nonlinear Differential Equations.}
\newblock Cambridge University Press, New York, 1994.

\bibitem{Va80}
S.J. van Strien.
\newblock On the bifurcations creating horseshoes.
\newblock In D.A. Rand and L.-S. Young, editors, {\em Dynamical Systems and
  Turbulence, Warwick, 1980}, pages 316--351. Springer, New York, 1981.

\bibitem{DeVa93}
W.~de~Melo and S.~van Strien.
\newblock {\em One-Dimensional Dynamics.}
\newblock Springer-Verlag, New York, 1993.

\bibitem{PiGr91}
A.S. Pikovsky and P.~Grassberger.
\newblock Symmetry breaking bifurcation for coupled chaotic attractors.
\newblock {\em J. Phys. A: Math. Gen.}, 24:4587--4597, 1991.

\bibitem{Gl01}
P.~Glendinning.
\newblock Milnor attractors and topological attractors of a piecewise linear
  map.
\newblock {\em Nonlinearity}, 14(2):239--257, 2001.

\bibitem{Lo78}
R.~Lozi.
\newblock Un attracteur \'{e}trange(?) du type attracteur de {H}\'{e}non.
\newblock {\em J. Phys. (Paris)}, 39(C5):9--10, 1978.
\newblock In French.

\bibitem{Mi80}
M.~Misiurewicz.
\newblock Strange attractors for the {L}ozi mappings.
\newblock In R.G. Helleman, editor, {\em Nonlinear dynamics, Annals of the New
  York Academy of Sciences}, pages 348--358, 1980.

\bibitem{NuYo92}
H.E. Nusse and J.A. Yorke.
\newblock Border-collision bifurcations including ``period two to period
  three'' for piecewise smooth systems.
\newblock {\em Phys. D}, 57:39--57, 1992.

\bibitem{Si16}
D.J.W. Simpson.
\newblock Border-collision bifurcations in $\mathbb{R}^n$.
\newblock {\em SIAM Rev.}, 58(2):177--226, 2016.

\bibitem{AvSc12}
V.~Avrutin, M.~Schanz, and S.~Banerjee.
\newblock Occurrence of multiple attractor bifurcations in the two-dimensional
  piecewise linear normal form map.
\newblock {\em Nonlin. Dyn.}, 67:293--307, 2012.

\bibitem{Si10}
D.J.W. Simpson.
\newblock {\em Bifurcations in Piecewise-Smooth Continuous Systems.}, volume~70
  of {\em Nonlinear Science}.
\newblock World Scientific, Singapore, 2010.

\bibitem{SuGa10}
I.~Sushko and L.~Gardini.
\newblock Degenerate bifurcations and border collisions in piecewise smooth
  1{D} and 2{D} maps.
\newblock {\em Int. J. Bifurcation Chaos}, 20(7):2045--2070, 2010.

\bibitem{Os06}
H.M. Osinga.
\newblock Boundary crisis bifurcation in two parameters.
\newblock {\em J. Diff. Eq. Appl.}, 12(10):997--1008, 2006.

\bibitem{PaTa93}
J.~Palis and F.~Takens.
\newblock {\em Hyperbolicity and sensitive chaotic dynamics at homoclinic
  bifurcations.}
\newblock Cambridge University Press, New York, 1993.

\bibitem{Si16b}
D.J.W. Simpson.
\newblock Unfolding homoclinic connections formed by corner intersections in
  piecewise-smooth maps.
\newblock {\em Chaos}, 26:073105, 2016.

\bibitem{AlPu17}
J.F. Alves, A.~Pumari\~{n}o, and E.~Vigil.
\newblock Statistical stability for multidimensional piecewise expanding maps.
\newblock {\em Proc. Amer. Math. Soc.}, 145(7):3057--3068, 2017.

\end{thebibliography}

\end{document}